\documentclass{amsproc}

\usepackage[T1]{fontenc}
\usepackage[utf8]{inputenc}

\usepackage{hyperref}
\usepackage{subfig}

\usepackage[english]{babel}

\usepackage{tikz}
\usetikzlibrary{decorations.markings}

\usepackage{enumitem}

\usepackage{breqn}

\copyrightinfo{2014}{American Mathematical Society}

\newtheorem{theorem}{Theorem}[section]

\theoremstyle{definition}

\theoremstyle{remark}

\numberwithin{equation}{section}

\DeclareMathOperator{\Tr}{Tr}
\DeclareMathOperator{\Id}{Id}
\DeclareMathOperator{\Ch}{Ch}
\newcommand{\Sym}[1]{S_{#1}}

\newcommand{\pp}{\mathbf{p}}
\newcommand{\qq}{\mathbf{q}}

\newtheorem{corollary}[theorem]{Corollary}
\newtheorem{conjecture}[theorem]{Conjecture}

\begin{document}

\title{Stanley character polynomials}
\author{Piotr \'Sniady}
\address{
Adam Mickiewicz University,
The Faculty of Mathematics and Computer Science, 
Collegium Mathematicum, 
ul.~Umultowska 87, 61-614 Poznań, Poland
 \newline
\indent 
Institute of Mathematics, Polish Academy of Sciences, 
\mbox{ul.~\'Sniadec\-kich 8,} \linebreak 00-956 Warszawa, Poland
} 
\email{piotr.sniady@amu.edu.pl}
\thanks{}

\dedicatory{Dedicated to Richard P.~Stanley on the occasion of his seventieth birthday.} 



\begin{abstract} 
Stanley considered suitably normalized characters 
of the symmetric groups on Young diagrams having a special geometric form, namely \emph{multirectangular Young diagrams}.
He proved that the character is a polynomial in the lengths of the sides of 
the rectangles forming the Young diagram and he
conjectured an explicit form of this polynomial.
This \emph{Stanley character polynomial} and this way of parametrizing the set of  Young diagrams
turned out to be a powerful tool for several problems of the dual combinatorics
of the characters of the symmetric groups and asymptotic representation theory, in particular to Kerov polynomials. 
\end{abstract}

\maketitle

We shall review Stanley's contribution to understanding of the \emph{normalized characters}
$\Ch_\pi(\pp\times \qq)$ of the symmetric groups corresponding to \emph{multirectangular Young diagrams} $\pp\times \qq$, 
see Figure~\ref{fig:multirectangular}. We will be guided by the following motivating example:
\begin{dmath}
\label{eq:motivating}
\Ch_5( \pp \times \qq)=
p_{1}^{5} q_{1} - 10 p_{1}^{4} q_{1}^{2} + 20 p_{1}^{3} q_{1}^{3} - 10
p_{1}^{2} q_{1}^{4} + \mathbf{1} p_{1} q_{1}^{5} + 15 p_{1}^{3} q_{1} - 40
p_{1}^{2} q_{1}^{2} + \mathbf{15} p_{1} q_{1}^{3} + \mathbf{8} p_{1} q_{1}+
\cdots+25 p_1 p_2 p_3 q_1 q_2 q_3+\cdots
\end{dmath}
(the full expression in the case $\pp=(p_1,p_2,p_3)$, $\qq=(q_1,q_2,q_3)$ of three rectangles --- shown in Figure~\ref{fig:multirectangular} --- has more than a hundred summands).

\begin{figure}[t]
\begin{tikzpicture}[scale=0.6]

\begin{scope}
\clip (0,0) -- (11,0) -- (11,2) -- (8,2) -- (8,5) -- (4,5) -- (4,7) -- (0,7) -- cycle;
\draw[black!20] (0,0) grid[step=1] (20,10); 
\end{scope}

\draw[ultra thick](0,0) -- (11,0) -- (11,2) -- (8,2) -- (8,5) -- (4,5) -- (4,7) -- (0,7) -- cycle;

\draw[dotted,thick] (8,2) -- (0,2)
(4,5) -- (0,5);

\begin{scope}[<->,thick,auto,dashed]
\draw (11.3,0) to node[swap] {$p_1$} (11.3,2);
\draw (0,1.7) to node[swap,shape=circle,fill=white] {$q_1$} (11,1.7);

\draw (8.3,2) to node[swap] {$p_2$} (8.3,5);
\draw (0,4.7) to node[swap,shape=circle,fill=white] {$q_2$} (8,4.7);

\draw (4.3,5) to node[swap] {$p_3$} (4.3,7);
\draw (0,6.7) to node[swap,shape=circle,fill=white] {$q_3$} (4,6.7);
\end{scope}

\end{tikzpicture}

\caption{Multirectangular Young diagram $\pp\times \qq$ with $\pp=(2,3,2)$ and $\qq=(11,8,5)$.}
\label{fig:multirectangular}
\end{figure}
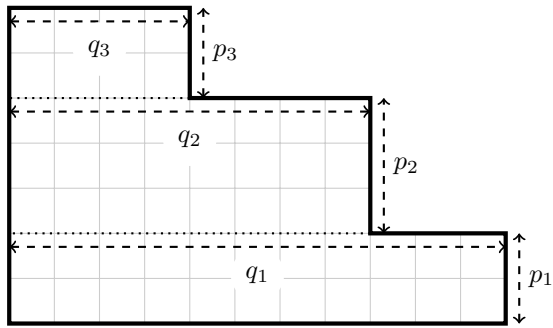

\section{Normalized characters of the symmetric groups}
\label{sec:normalized-characters}
 
In the study of the irreducible characters of the symmetric groups
\[ \chi^\lambda(\pi):= \Tr \rho^\lambda(\pi) \]
the usual viewpoint is to fix the irreducible representation $\rho^\lambda$ and to view the character 
as a function of (the conjugacy class of) the permutation $\pi$. It was a brilliant idea of Kerov and Olshanski \cite{KerovOlshanski1994}
to do roughly the opposite: fix the conjugacy class and to consider the character as a function on the set of Young diagrams.
In order for this \emph{dual approach} to be successful one has to choose the right normalization of characters, see below.

If $\pi$ is a partition of $k$ and $\lambda$ is a partition of $n$, the corresponding \emph{normalized character} is defined as
\begin{equation}
\label{eq:normalized-characters}
\Ch_\pi(\lambda):=
\begin{cases}
   \underbrace{n (n-1) \cdots (n-k+1)}_{\text{$k$ factors}} \frac{\displaystyle \chi^\lambda(\pi 1^{n-k})}{\displaystyle \chi^\lambda(1^n)} &
\text{if $n\geq k$}, \\
   0 & \text{otherwise},
\end{cases}
\end{equation}
where for the evaluation of the character $\chi^\lambda(\pi 1^{n-k})$ we identify the partition $\pi 1^{n-k}$ with an arbitrary permutation in the symmetric group $\Sym{n}$ with the appropriate cycle decomposition.

As we mentioned at the beginning, we shall review Stanley's contribution to understanding of such normalized characters
$\Ch_\pi$. In the following section we will present the context of the asymptotic representation theory where such characters play a very prominent role;
an impatient reader may jump directly to Section~\ref{sec:character-polynomial}.

\section{Asymptotic representation theory of symmetric groups}
\label{sec:asymptotic-representatation}

The above normalization \eqref{eq:normalized-characters} of the characters was chosen by Kerov and Olshanski \cite{KerovOlshanski1994}
in such a way, that for each partition $\pi$ the corresponding character 
$\lambda\mapsto\Ch_\pi(\lambda)$ can be expressed 
as a \emph{universal} multivariate polynomial in some simple
functionals $\big(\lambda\mapsto F_\alpha(\lambda) \big)_\alpha$ of shape of the Young diagram $\lambda$:
\begin{equation}
\label{eq:universal}
 \Ch_\pi= P_\pi (  F_\alpha ).  
\end{equation}
An example is postponed until Section~\ref{sec:free-cumulants}.
We say that the polynomial $P_\pi$ is \emph{universal} because it does not depend on the choice of $\lambda$.

There is some freedom for choosing such functionals of shape and several useful choices are available.
For applications it is important that it is possible to choose the functionals $F_\alpha(\lambda)$ in such a way 
that they describe the \emph{``macroscopic shape''} of $\lambda$ in some convenient way.

It should be stressed that equalities of type \eqref{eq:universal} are very appealing from the viewpoint of the 
\emph{asymptotic representation theory of symmetric groups} because they provide a direct link between
the values of the irreducible characters and the macroscopic shape of $\lambda$; a link that remains useful in the limit
when the number of boxes of $\lambda$ tends to infinity. 
This is quite opposite to the Murnaghan-Nakayama rule, the classical tool for calculating characters, 
for which the number of terms grows very quickly with the number of boxes of $\lambda$ and, due to cancellations, it is hard to obtain a meaningful asymptotic answer.

The link provided by equalities of type \eqref{eq:universal} between the characters and the macroscopic shape of $\lambda$
was the cornerstone of some upper bounds on the characters
\cite{RattanSniady2008} as well as 
several results concerning the statistical properties of the random irreducible 
component $\rho^\lambda$ of a given reducible representation.
These results include 
the law of large numbers for the shape of the corresponding random Young diagram~$\lambda$ \cite{Biane1998} 
and central limit theorem for the fluctuations of $\lambda$ 
around this shape \cite{Kerov1993,IvanovOlshanski2002,Sniady2006}.

Unfortunately, this result \eqref{eq:universal} of Kerov and Olshanski \cite{KerovOlshanski1994} is only existential
and it does not say too much on how exactly this polynomial $P_\pi$ looks like.
\emph{How to find such a polynomial explicitly?} 
This question is motivated not only by some algebraic-combinatorial curiosity,
since in all the papers cited in the previous paragraph 
the key step was to find some upper bounds on the asymptotic behavior of this polynomial $P_\pi$.

The results of Stanley which we discuss in this note are related to the problem of finding this polynomial explicitly
in the case when for the functionals of shape $(F_\alpha)$ we take \emph{the free cumulants}. We will review these quantities in the following section.

\section{Free cumulants and Kerov polynomials}
\label{sec:free-cumulants}

Biane \cite{Biane1998} (for a quick overview article see \cite{Biane2001a}) 
introduced a particularly nice family of functionals of the shape of a Young diagram,
namely \emph{the free cumulants} $R_2(\lambda), R_3(\lambda), \dots$. 
Their original definition (\emph{``Speicher's free cumulants
of Kerov's transition measure 
of the Young diagram $\lambda$''})
is somewhat technically involved.
For our purposes it is enough to keep in mind that this definition gives a concrete 
and computationally efficient relationship between the macroscopic shape of a Young diagram $\lambda$ 
and the sequence of its free cumulants $R_2(\lambda), R_3(\lambda),\dots$.

Their key feature is the following one: 
\emph{the free cumulant $R_{k+1}$ gives the first-order approximation for the value of the character $\Ch_k$
on the cycle of length $k$}, thus free cumulants give an (approximate) answer to the fundamental problem of the relationship 
between the characters and the macroscopic shape of the Young diagram. 
In the above statement we have to specify in which way the Young diagram should tend to infinity 
as there are many choices available.
In this note we are interested in the limit in which the Young diagram tends to infinity, maintaining its macroscopic shape; 
this kind of limit can be obtained by considering \emph{dilations} $s\lambda$ for $s\to\infty$ and for a fixed Young diagram $\lambda$, see Figure~\ref{fig:youngB}.
The above-mentioned key feature of the free cumulants can be formulated precisely as follows:
 for any $k\geq 1$ and any Young diagram $\lambda$
\begin{equation}
\label{eq:def-free-cumulants}
 R_{k+1}(\lambda) = \lim_{s\to\infty} \frac{\Ch_{k}(s\lambda)}{s^{k+1}}. 
\end{equation}
In fact, the above equality can be regarded as an alternative (but a bit implicit) definition of free cumulants.

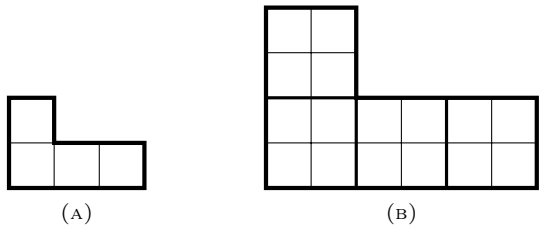
\begin{figure}[t]
\centering
\subfloat[][]{
\begin{tikzpicture}[scale=0.6]
\footnotesize

\draw[ultra thick] (0,0) -- (3,0) -- (3,1) -- (1,1) -- (1,2) -- (0,2) -- cycle; 
\clip (0,0) -- (3,0) -- (3,1) -- (1,1) -- (1,2) -- (0,2); 
\draw (0,0) grid (3,3);
\end{tikzpicture}
\label{fig:youngA}
} \hspace{8ex}
\subfloat[][]{
\begin{tikzpicture}[scale=1.2]
\footnotesize

\draw[ultra thick] (0,0) -- (3,0) -- (3,1) -- (1,1) -- (1,2) -- (0,2) -- cycle; 
\clip (0,0) -- (3,0) -- (3,1) -- (1,1) -- (1,2) -- (0,2); 
\draw (0,0) grid[step=0.5] (3,3);
\draw[very thick] (0,0) grid (3,3);
\end{tikzpicture} 
\label{fig:youngB}
}

\caption{\protect\subref{fig:youngA} Young diagram $\lambda=(3,1)$. 
\protect\subref{fig:youngB}
The dilation $2\lambda=(6,6,2,2)$ of the Young diagram $\lambda$ shown on the left.
Informally speaking, the original Young diagram and its dilation have the same macroscopic shape; they differ
just by the number of boxes.}
\end{figure}

The expansion \eqref{eq:universal} in terms of free cumulants was explicitly introduced by Kerov 
\cite{Kerov2000talk} (it appeared implicitly before, in one of the proofs of Biane \cite{Biane1998}),
in particular Kerov calculated the first few examples for the case when $\pi=(k)$ consists of only one part:
 \begin{align}
\label{eq:Kerov}
\notag \Ch_1 &= R_2, \\
\notag \Ch_2 &= R_3, \\
\notag \Ch_3 &= R_4 + R_2,   \\
\notag \Ch_4 &= R_5 + 5R_3,    \\     
\Ch_5 &=\mathbf{1} R_6 + \mathbf{15} R_4 + 5R_2^2 + \mathbf{8} R_2.
\end{align}
The right-hand sides are now called \emph{Kerov character polynomials}; we denote them by the special symbol
$K_k=K_k(R_2,R_3,\dots)$.
Here and in \eqref{eq:universal} we use a simplified notation and write $\Ch_\pi$ instead of $\Ch_\pi(\lambda)$
as well as we write $R_k$ instead of $R_k(\lambda)$; the equalities hold true for an arbitrary Young diagram $\lambda$.

Kerov \cite{Kerov2000talk} proved that the coefficients of these polynomials are integers
(for the proof see \cite{Biane2003}; this reference is also a good introduction to the topic of Kerov polynomials) 
and he conjectured that these coefficients are \emph{non-negative}. 
If this conjecture --- called \emph{Kerov positivity conjecture} --- is true, it is natural to ask:
\emph{are there some natural combinatorial objects,
the cardinalities of which are counted by the coefficients of Kerov polynomials?}
Thus Kerov's conjecture hinted on the existence of some mysterious hidden combinatorial structure behind the 
characters of the symmetric groups. 
For this reason it was a powerful force influencing the research in algebraic combinatorics
for almost a decade. 

The expansion of the character $\Ch_5$ in terms of the free cumulants \eqref{eq:Kerov}
and the expansion of this character on a multirectangular Young diagram \eqref{eq:motivating} 
have some coefficients in common, namely the numbers $1$, $15$ and $8$ (marked boldface in both equations).
This is not an accident; in fact it has been observed by Stanley \cite{Stanley2003/04} 
that the knowledge of the coefficients 
of the polynomial $\Ch_\pi(\pp\times \qq)$ can be used in order to calculate (some of) the coefficients of Kerov polynomials. 
We will review this link in the following section.

\section{Characters on rectangular Young diagrams}
\label{sec:rectangular}

\begin{figure}[t]
\begin{tikzpicture}[scale=0.7]

\begin{scope}
\draw[black!20] (0,0) grid[step=1] (5,4); 
\end{scope}

\draw[ultra thick] (0,0) rectangle (5,4);

\begin{scope}[<->,thick,auto,dashed]
\draw (4.3,0) to node[shape=circle,fill=white] {$p$} (4.3,4);
\draw (0,3.3) to node[swap,shape=circle,fill=white] {$q$} (5,3.3);
\end{scope}

\end{tikzpicture}

\caption{Rectangular Young diagram $p\times q$ with $p=4$ and $q=5$.}
\label{fig:rectangular}
\end{figure}
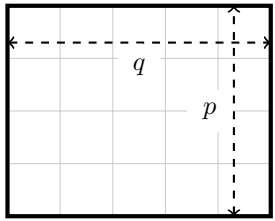

Stanley \cite{Stanley2003/04} considered rectangular Young diagrams 
$p\times q=(\underbrace{q,q,\dots,q}_{\text{$p$ times}})$, see
Figure~\ref{fig:rectangular}.
He proved the following partial result which concerns such diagrams.

\begin{theorem}[{\cite[Theorem 1]{Stanley2003/04}}]  
\label{thm:small-stanley}
Let $\pi\in \Sym{k}$. 
Then
\begin{equation}
\label{eq:stanley-simple}
 \Ch_\pi(p \times q) = \sum_{\substack{\sigma_1,\sigma_2\in \Sym{k} \\  \sigma_1 \sigma_2=\pi}}  
(-1)^{\sigma_1}\; q^{\kappa(\sigma_1)}\; p^{\kappa(\sigma_2)}, 
\end{equation}
where $\kappa(\alpha)$ denotes the number of the cycles of a permutation $\alpha\in \Sym{n}$,
and where for the purposes of the normalized character $\Ch_\pi$ we view $\pi\vdash k$ as a partition 
which gives the cycle structure of $\pi\in \Sym{k}$.
\end{theorem}

The above formula for the character
might not be very useful from the practical point of view 
(the number of summands grows very quickly with $k$).
Nevertheless, it is very convenient for several problems of the asymptotic representation theory 
in which one studies the character $\Ch_\pi$ for a fixed conjugacy class $\pi$ 
while the Young diagram $\lambda$ (i.e., $p$ and $q$) tends to infinity, 
because the number of terms, no matter how large,
remains constant and each term has a very simple form.
It was one of the first few explicit formulas for the characters $\Ch_\pi$ with this property.
Its drawback is that it gives the characters only for Young diagrams of some specific form,
nevertheless it could still be used to find some new bounds for the characters on Young diagrams of arbitrary shape
\cite{RattanSniady2008}.

This theorem together with the definition \eqref{eq:def-free-cumulants} implies that the free cumulant 
$R_k(p\times q)$ is equal to the the sum of these terms on the right-hand side of \eqref{eq:stanley-simple}
applied to $\pi=(1,2,\dots,k)$ for which $\kappa(\sigma_1)+\kappa(\sigma_2)=k+1$
(in the literature one says that $\pi=\sigma_1 \sigma_2$ is a \emph{minimal factorization}). 
In particular,
\begin{equation}
\label{eq:free-cumulants-first-approx}
 R_{k+1}(p\times q) = p q^{k}+ \text{(summands which are divisible by $p^2$)}.    
\end{equation}
The latter equality implies that the coefficients of the linear terms of Kerov polynomials are directly related to some coefficients of the polynomial
$\Ch_k(p\times q)$:
\begin{equation}
\label{eq:kerov-linear}
 [R_{i} ] K_k = [p q^{i-1} ] \Ch_k(p\times q).   
\end{equation}
Stanley's formula \eqref{eq:stanley-simple} gives an explicit combinatorial interpretation to the right-hand-side,
thus we obtain the following partial result which supports Kerov positivity conjecture.
This result appears implicitly in the work of Stanley
\cite[pp.~9--10]{Stanley2003/04} and explicitly in an independent work of Biane \cite[Theorem 6.1]{Biane2003}
who used a closely related method.

\begin{corollary}
\label{coro:linear-kerov}
The linear term $[R_i] K_k$ of the Kerov polynomial is equal to the number of pairs 
$(\sigma_1,\sigma_2)\in \Sym{k}\times \Sym{k}$ such that $\sigma_1 \sigma_2=(1,2,\dots,k)$, and 
the number of cycles of $\sigma_1$ is equal to $i-1$ while the number of cycles of $\sigma_2$ is equal to~$1$.    
\end{corollary}

The above reasoning explains the coincidences between some of the coefficients of \eqref{eq:Kerov} and \eqref{eq:motivating}. 
Unfortunately, the information about the behavior of the characters only on the rectangular Young diagrams
is not sufficient for most purposes, including study of more complicated coefficients of the Kerov polynomials.
In the following section we will discuss Stanley's attempt to overcome this difficulty.

\section{Stanley character polynomial}
\label{sec:character-polynomial}

In the same paper \cite{Stanley2003/04} Stanley introduced an interesting class of  
\emph{multirectangular Young diagrams $\pp\times \qq$}  
(where $\pp=(p_1,\dots,p_\ell)$ is a tuple of non-negative integers 
and $\qq=(q_1,\dots,q_\ell)$ is a non-increasing tuple of non-negative integers, 
see Figure~\ref{fig:multirectangular}) 
and initiated the investigation of the corresponding character $\Ch_\pi(\pp\times \qq)$.
The numbers $p_1,p_2,\dots,q_1,q_2,\dots$ are now called \emph{Stanley coordinates}.
One could argue that the introduction of the class of multirectangular Young diagrams was rather naive, as
\emph{any} Young diagram $\lambda=(\lambda_1,\lambda_2,\dots)$ 
can be regarded as a multirectangular one for the trivial choice $\pp=(\lambda_1,\lambda_2,\dots)$, $\qq=(1,1,\dots)$.
However, the core of Stanley's idea  was the possibility of  considering 
the character $\Ch_\pi(\pp\times \qq)$ for a fixed conjugacy class $\pi$,
as a \emph{polynomial} in the variables
$p_1,p_2,\dots,q_1,q_2,\dots$ and studying its coefficients.
For an example see \eqref{eq:motivating}.
This polynomial, now called \emph{the Stanley character polynomial}, turned out to be a powerful algebraic tool for
investigation of the characters of the symmetric groups.

Note that the roles played by indeterminates $\pp=(p_1,\dots,p_\ell)$ are very different than the roles played by $\qq=(q_1,\dots,q_\ell)$; in particular the multirectangular Young diagram $\qq\times\pp$ might not be
well-defined. 
The reader may also complain that the description of the fundamental involution of Young diagrams
$\lambda \mapsto \lambda'$ given by transposition is rather awkward in Stanley coordinates
and thus one should rather find another parametrization of multirectangular Young diagrams
which would make this symmetry more transparent; for example some new parametrization in which 
$(\pp \cdot \qq)'=\qq\cdot \pp$.
However, it seems that
the calculations such as the one from Section \ref{sec:kerov-toy-example} work best in the original Stanley coordinates
and not in such more democratic variations.

Stanley's investigation culminated in a truly remarkable unpublished preprint \cite{Stanley2006},
in which the following conjectural explicit formula \eqref{eq:stanley-general} 
for the Stanley character polynomial was stated. 
This formula is similar to the double sum over the symmetric group from \eqref{eq:stanley-simple},
but it additionally involves a summation over \emph{``colorings''} of the cycles
$f_1\colon C(\sigma_1)\rightarrow\{1,\dots,\ell\}$ and $f_2\colon C(\sigma_2)\rightarrow\{1,\dots,\ell\}$,
where $C(\alpha)$ denotes the set of cycles of $\alpha\in\Sym{k}$.
We sum over all possible colorings $f_2$ of the cycles of $\sigma_2$, 
while the coloring $f_1$ of the cycles of $\sigma_1$ is determined.

\begin{conjecture}[{\cite[Conjecture 3]{Stanley2006}}]
\label{conj:stanley-conjecture}
For any $\pi\in\Sym{k}$
\begin{multline}
\label{eq:stanley-general}
 \Ch_\pi(\pp\times \qq) \\ = \sum_{\substack{\sigma_1,\sigma_2\in \Sym{k} \\  \sigma_1 \sigma_2=\pi}}
\; \sum_{f_2\colon C(\sigma_2)\rightarrow\{1,2,\dots,\ell\}}   
(-1)^{\sigma_1}\; \left( \prod_{c\in C(\sigma_1)} q_{f_1(c)} \right) 
\left( \prod_{d\in C(\sigma_2)} p_{f_2(d)} \right),
\end{multline}
where the coloring
$f_1\colon C(\sigma_1)\rightarrow\{1,\dots,\ell\}$ is defined by
\[ f_1(c) = \max_{\substack{d\in C(\sigma_2), \\ c \cap d \neq \emptyset}}
  f_2(d);\]
the maximum is taken over all cycles $d\in C(\sigma_2)$ which have a non-empty intersection with the cycle
$c\in C(\sigma_1)$.
\end{conjecture}

The above formula, which is currently called \emph{Stanley character formula},
might at first sight scare the Reader; in the following section we will 
get familiar with it by analyzing a toy example, while in Section~\ref{sec:stanley-embedding}
we will present it in an equivalent, but more transparent form.

\section{Stanley character polynomials in action: \\ toy example of Kerov polynomials}
\label{sec:kerov-toy-example}

A monomial in the variables $p_1,\dots,p_\ell,q_1,\dots,q_\ell$ is called \emph{$\pp$-square-free} if it is not divisible by 
$p_i^2$ for any $1\leq i\leq \ell$. It turns out that such $\pp$-square-free terms of Stanley character polynomials 
encode in a convenient way the information about Kerov polynomials. In this section
we will explore this connection in a 
toy example. The reasoning presented in this section is a simplified version of the one from \cite{DolpolhkFeraySniady2009}.

In the following we shall assume the validity of Conjecture \ref{conj:stanley-conjecture}.
Using the topological tools which will be discussed in Section \ref{sec:stanley-embedding} one can easily show
the following extension of \eqref{eq:free-cumulants-first-approx} for $\pp=(p_1,p_2)$, $\qq=(q_1,q_2)$:
\begin{multline} 
\label{eq:free-cumulant-multirectangular}
R_{k+1}(\pp\times \qq) = p_1 q_1^{k} +p_2 q_2^{k} -
k \sum_{\substack{j_1\geq 1\\ j_2\geq 2 \\ j_1+j_2=k+1}} p_1 p_2 q_1^{j_1-1} q_2^{j_2-1} 
\\
+ \text{(summands which are divisible by $p_1^2$ or by $p_2^2$)}. 
\end{multline}

Assume that $F$ is a function on the set of Young diagrams which can be expressed as a polynomial in the free cumulants 
$R_2,R_3,\dots$. Then $F$ evaluated on the multirectangular Young diagram $\pp\times \qq$ becomes a polynomial in Stanley coordinates
$p_1,\dots,p_\ell,q_1,\dots,q_\ell$.
By a direct calculation based on \eqref{eq:free-cumulant-multirectangular} and by considering separately the three cases: 
(i) when $F=R_{k+1}$ is linear in free cumulants, (ii) when $F=R_{k_1+1} R_{k_2+1}$ is quadratic, and, 
(iii) when $F$ is a product of at least three free cumulants, it follows by linearity that for an arbitrary choice of $F$:
\begin{equation}
\label{eq:stanley-polynomial-symmetry}
\left[p_1 p_2 q_1^{j_1-1} q_2^{j_2-1} \right] F(\pp\times \qq) = \left[p_1 p_2 q_1^{j_2-1} q_2^{j_1-1} \right] F (\pp\times \qq) \qquad \text{for } j_1,j_2\geq 2,
\end{equation}
i.e., if the exponents of $q_1$ and $q_2$ are strictly positive, the coefficient does not change if we swap them.

Again by \eqref{eq:free-cumulant-multirectangular} and by considering separately the same three cases as above,
it follows for $j_1\neq j_2$ that the coefficient of the ``Kerov polynomial'' for $F$ is given by:
\begin{equation}
\label{eq:kerov-polynomial-toy-example}
[R_{j_1} R_{j_2}] F =\left[p_1 p_2 q_1^{j_1-1} q_2^{j_2-1}\right] F(\pp\times \qq) - 
\left[p_1 p_2 q_2^{j_1+j_2-2}\right] F(\pp\times \qq);    
\end{equation}
compare with \eqref{eq:kerov-linear}.
In particular, \eqref{eq:stanley-polynomial-symmetry} and \eqref{eq:kerov-polynomial-toy-example}
can be applied in the special case when $F=\Ch_k$ is the normalized character of the symmetric group;
in this case the left-hand-side of \eqref{eq:kerov-polynomial-toy-example} is just the coefficient $[R_{j_1} R_{j_2}] K_k$
of the usual Kerov polynomial.

These results on $\pp$-square-free terms of Stanley character polynomials are sufficient to find a combinatorial
interpretation of the quadratic coefficients of Kerov polynomials; however this time the proof will be much more
complicated than the one of Corollary \ref{coro:linear-kerov}.
This result was first proved by F\'eray \cite[Theorem 1.4.4]{Feray2010}.

\begin{theorem}
\label{theo:quadratic-reformulated}
For all integers $j_1\neq j_2$ such that $j_1,j_2\geq 2$ and $k\geq 1$ the quadratic coefficient $[R_{j_1} R_{j_2}] K_{k}$ of
the Kerov polynomial is equal to the number of triples $(\sigma_1,\sigma_2,f)$ with the following
properties:
\begin{enumerate}[label=(\alph*)]
 \item \label{enum:quadratic-a} 
$\sigma_1,\sigma_2\in \Sym{k}$ are such that $\sigma_1 \sigma_2=(1,2,\dots,k)$;
 \item $\sigma_2$ consists of two cycles;
 \item $\sigma_1$ consists of $j_1+j_2-2$ cycles;
 \item 
\label{enum:quadratic-c} 
$f\colon C(\sigma_2)\rightarrow \{1,2\}$ is a bijective labeling of the two cycles
of $\sigma_2$;
 \item for each cycle $c\in C(\sigma_2)$ there are at least $j_{f(c)}$ cycles
of $\sigma_1$ which intersect $c$ nontrivially.
\end{enumerate}
\end{theorem}

\begin{proof}
Let us compute the number of the triples
$(\sigma_1,\sigma_2,f)$ which are counted in the statement of the theorem. By the inclusion-exclusion principle it is equal
to 
\begin{multline}
\label{eq:paskuda}
\big(\text{\#triples which fulfill
conditions \ref{enum:quadratic-a}--\ref{enum:quadratic-c}}\big) \\
+(-1) \big(\text{\#triples for which the cycle $f^{-1}(1)$ intersects
at most $j_1-1$ cycles of $\sigma_1$}\big)  \\
+(-1) \big(\text{\#triples for which the cycle $f^{-1}(2)$ intersects
at most $j_2-1$ cycles of $\sigma_1$}\big). 
\end{multline}
It might seem that the above formula is not complete, since the inclusion-exclusion principle
has one term more; however one can easily check that in this case it vanishes.

By the Stanley character formula \eqref{eq:stanley-general} the first
summand of \eqref{eq:paskuda} is equal to 
\begin{equation}
\label{eq:rown-a}
(-1) \sum_{\substack{a+b=j_1+j_2-2,\\ 1\leq b}} \left[p_1 p_2 q_1^{a} q_2^{b} \right] \Ch_k(\pp\times \qq),
\end{equation}
the second summand of \eqref{eq:paskuda} is equal to 
\begin{equation}
\label{eq:rown-b}
  \sum_{\substack{a+b=j_1+j_2-2,\\ 1\leq a\leq j_1-1}} 
\left[p_1 p_2 q_1^{b} q_2^{a}\right] \Ch_k(\pp\times \qq), 
\end{equation}
and the third summand of \eqref{eq:paskuda} is equal to 
\begin{equation}
\label{eq:rown-c}
\sum_{\substack{a+b=j_1+j_2-2,\\ 1\leq b\leq j_2-1}}
\left[p_1 p_2 q_1^{a} q_2^{b} \right] \Ch_k(\pp\times \qq).
\end{equation}

We can apply \eqref{eq:stanley-polynomial-symmetry} to the
summands of \eqref{eq:rown-b}; it follows that
\eqref{eq:rown-b} is equal to 
\begin{equation}
 \label{eq:rown-d}
  \sum_{\substack{a+b=j_1+j_2-2,\\ 1\leq a\leq j_1-1}}
\left[p_1 p_2 q_1^{a} q_2^{b}  \right] \Ch_k(\pp\times \qq).
\end{equation}

The sum of \eqref{eq:rown-a}, \eqref{eq:rown-d}, and \eqref{eq:rown-c} 
is equal to the right-hand of \eqref{eq:kerov-polynomial-toy-example} which finishes
the proof.
\end{proof}

\section{Stanley character polynomials in action: \\ generic coefficients of Kerov polynomials}

The pattern which we encountered for the linear and for the quadratic terms of Kerov polynomials
turns out to hold true in general; we shall review it briefly.
It turns out that an arbitrary coefficient of the Kerov polynomial $K_k$ can be expressed
as a closed formula in terms of the coefficients of the Stanley polynomial $\Ch_k(\pp\times \qq)$ 
(for the special case of the linear terms see \eqref{eq:kerov-linear}; for the quadratic terms see 
\eqref{eq:kerov-polynomial-toy-example}).
Because of some symmetries of the Stanley character polynomial (analogous to \eqref{eq:stanley-polynomial-symmetry})
this closed formula can be rewritten in a form which allows a direct combinatorial interpretation
(for the special case of quadratic terms see Theorem \ref{theo:quadratic-reformulated} and its proof).
For the generic case this combinatorial interpretation has a flavour related to Theorem \ref{theo:quadratic-reformulated},
but it is much more complicated \cite{DolegaFeraySniady2010}. This gives a proof of the Kerov's positivity conjecture
together with an explicit combinatorial interpretation of the coefficients of Kerov polynomials.

Note that the first proof of Kerov's positivity conjecture  
(but without the explicit combinatorial interpretation)
was earlier, due to F\'eray \cite{Feray2009}.

\section{The structure behind Stanley character formula}
\label{sec:stanley-embedding}

\begin{figure}
\centering
\subfloat[][]{\begin{tikzpicture}[scale=0.5,
white/.style={circle,draw=black,inner sep=4pt},
black/.style={circle,draw=black,fill=black,inner sep=4pt},
connection/.style={draw=black!80,black!80,auto}
]
\footnotesize

\begin{scope}
\clip (0,0) rectangle (10,10);

\draw (3.333,2.333) node (b1)    [black,label=90:$\Pi$] {};
\draw (b1) +(10,0) node (b1prim) [black] {};

\draw (5,7.333) node (b2)     [black,label=0:$\Sigma$] {};
\draw (b2) +(0,-10) node (b2prim) [black] {};

\draw (b2) +(-3,-1) node (w2) [white,label=180:$W$] {};

\draw (6.666,3.333) node (w1) [white,label=45:$V$] {};
\draw (w1) +(-10,0) node (w1left) [white] {};
\draw (w1) +(0,10)  node (w1top)  [white] {};

\draw[connection] (b1) to node {$4$} node [swap] {} (w1);

\draw[connection] (b2) to node {$3$} node [swap] {} (w2);

\draw[connection,pos=0.2] (b2) to node {$2$} node [swap] {} (w1top);
\draw[connection,pos=0.7] (b2prim) to node {$2$} node [swap] {} (w1);

\draw[connection,pos=0.666] (b1prim) to node {$1$} node [swap] {} (w1);
\draw[connection,pos=0.333] (b1) to node {$1$} node [swap] {} (w1left);

\draw[connection,pos=0.5] (w1) to node {} node [swap] {$5$} (b2);

\end{scope}

\draw[very thick,decoration={
    markings,
    mark=at position 0.5 with {\arrow{>}}},
    postaction={decorate}]  
(0,0) -- (10,0);

\draw[very thick,decoration={
    markings,
    mark=at position 0.5 with {\arrow{>}}},
    postaction={decorate}]  
(0,10) -- (10,10)  ;

\draw[very thick,decoration={
    markings,
    mark=at position 0.5 with {\arrow{>>}}},
    postaction={decorate}]  
(0,0) -- (0,10);

\draw[very thick,decoration={
    markings,
    mark=at position 0.5 with {\arrow{>>}}},
    postaction={decorate}]  
(10,0) -- (10,10)  ;

\end{tikzpicture}
\label{subfig:map}}
\hfill
\subfloat[][]{
\begin{tikzpicture}[scale=1.2]
\begin{scope}

\draw[line width=5pt,black!20] (-0.2,0.5) -- (3.2,0.5);
\draw (3.2,0.5) node[anchor=west] {$\Sigma$};

\draw[line width=5pt,black!20] (-0.2,1.5) -- (1.2,1.5);
\draw (1.2,1.5) node[anchor=west] {$\Pi$};

\draw[line width=5pt,black!20] (2.5,-0.2) -- (2.5,1.2);
\draw (2.5,1.2) node[anchor=south] {$W$};

\draw[line width=5pt,black!20] (0.5,-0.2) -- (0.5,2.2);
\draw (0.5,2.2) node[anchor=south] {$V$};

\draw[ultra thick] (3,0) -- (3,1) -- (1,1) -- (1,2) -- (0,2) -- (0,0) -- cycle; 
\clip (0,0) -- (3,0) -- (3,1) -- (1,1) -- (1,2) -- (0,2); 
\draw (0,0) grid (3,3);
\end{scope}
\draw (0.5,-0.2) node[anchor=north,text height=8pt] {$a$};
\draw (1.5,-0.2) node[anchor=north,text height=8pt] {$b$};
\draw (2.5,-0.2) node[anchor=north,text height=8pt] {$c$};

\draw (-0.2,0.5) node[anchor=east]  {$\alpha$};
\draw (-0.2,1.5) node[anchor=east] {$\beta$};

\draw(2.5,0.5) node {$3$};
\draw(0.5,0.5) node {$2,5$};
\draw(0.5,1.5) node {$1,4$};

\end{tikzpicture}
\label{subfig:embed}}

\caption{\protect\subref{subfig:map} Map on the torus;
the left side of the square should be glued to the right side,
as well as bottom to top, as indicated by arrows.
The white vertices correspond to the cycles of
$\sigma_1=(1,5,4,2)(3)$, the black vertices correspond to the cycles of $\sigma_2=(2,3,5)(1,4)$,
the unique face corresponds to the unique cycle of $\pi=\sigma_1 \sigma_2=(1,2,3,4,5)$. 
\protect\subref{subfig:embed} An example of an embedding of this map: 
$F(\Sigma)=\alpha$, $F(\Pi)=\beta$, $F(V)=a$, $F(W)=c$.
$F(1)=F(4)=(a \beta)$, $F(2)=F(5)=(a \alpha)$, $F(3)=(c\alpha)$.
The columns of the Young diagram were indexed by Latin letters, the rows by Greek letters.}
\label{fig:embedding}
\end{figure}
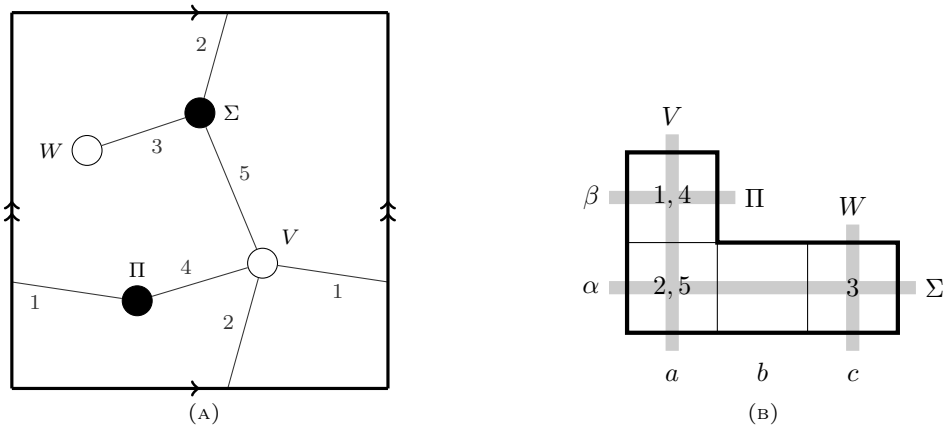

It seems that the true difficulty related to Stanley's character formula \eqref{eq:stanley-general}
was to guess its correct form. Indeed, once Conjecture \ref{conj:stanley-conjecture} has been formulated 
by Stanley \cite{Stanley2006}, 
a proof has been found by F\'eray \cite{Feray2010} only a half a year later
(some partial results have been obtained also by Rattan \cite{Rattan2008}).
Another half a year later, a more elementary proof has been found \cite{FeraySniady2011}.
In the same paper a combinatorial interpretation of the mysterious 
right-hand-side of Stanley's character formula \eqref{eq:stanley-general} has been found.
We will review it in the following.

Firstly, it is convenient to represent a pair of permutations $\sigma_1,\sigma_2\in\Sym{k}$ as an \emph{oriented map} \cite{LandoZvonkin2004}, i.e.,
as a bipartite graph with labeled edges embedded into an oriented surface, see Figure \ref{subfig:map}, 
determined as follows.
The cycles of the permutation $\sigma_1$ (resp., $\sigma_2$) determine 
the cyclic order (counterclockwise) of the labels around black (resp., white) vertices. 
The faces of the map correspond to the cycles of the product $\pi=\sigma_1 \sigma_2$.

An \emph{embedding} of such a map to a Young diagram $\lambda$ is a function which 
associates to white vertices -- columns of $\lambda$, to black vertices -- rows of $\lambda$, and to edges -- 
boxes of $\lambda$
(these functions need not to be injective). We also require that an embedding should preserve the incidence,
i.e.~if a vertex $V$ is incident to an edge $E$, then they should be mapped to a row or column $F(V)$ 
which contains the box $F(E)$, see Figure \ref{fig:embedding}.

It is rather easy to see that the right-hand side of Stanley character formula \eqref{eq:stanley-general} 
has a nice combinatorial interpretation as a 
(signed) sum of the numbers of embedding of all maps with a fixed face-structure $\pi=\sigma_1 \sigma_2$.

A very nice feature of this viewpoint on the Stanley character formula is that the asymptotic behavior of
each summand on a dilated Young diagram $s\lambda$ in the limit $s\to\infty$ depends directly on 
the Euler characteristic of the map. 
In the case of the characters $\Ch_k$ corresponding to a single cycle, 
the maximal contribution comes from \emph{planar maps}, i.e.~maps, which are drawn on the surface of a sphere.
This, together with the definition \eqref{eq:def-free-cumulants}, gives immediately a concrete
formula for free cumulants. We leave it as a simple exercise to the reader to recover 
\eqref{eq:free-cumulants-first-approx} and \eqref{eq:free-cumulant-multirectangular} in this way.

The Stanley character formula, viewed as in the current section, was the key in the proof of
new bounds on the character ratios $\frac{\chi^{\lambda}(\pi)}{\chi^{\lambda}(\Id)}$ in the case when the length of
the permutation $\pi\in\Sym{k}$ becomes large when compared to $k$ \cite{FeraySniady2011}.

\section{Further perspectives}
\label{sec:jack}

We will review some further advancements related to Stanley character polynomials and the Stanley character formula;
they indicate that the field remains active.

The characters of the symmetric groups can be alternatively viewed as the \emph{spherical functions}
of the \emph{Gelfand pair} $(\Sym{k}\times \Sym{k}, \Sym{k})$, where $\Sym{k}$ is viewed as a subgroup
of $\Sym{k}\times \Sym{k}$ via the diagonal map. It been observed by Scarabotti \cite{Scarabotti2011}
that an analogue of Stanley's formula
\eqref{eq:stanley-general} holds true as well for the Gelfand pair
$(\Sym{k}\times \Sym{k-1}, \Sym{k-1})$.

The paper \cite{AvalFerayNovelliThibon2013} concerns the most general form of polynomials which could arise as Stanley polynomials, as well as some non-commutative generalizations. 

Probably the most interesting direction of development in the field of Stanley polynomials 
is related to \emph{Jack characters}.
They are a one-parameter deformation of the characters of the symmetric groups,
and describe the dual combinatorics of \emph{Jack symmetric functions}.
Jack characters were introduced by Lassalle \cite{Lassalle2008,Lassalle2009}
who also formulated several challenging conjectures concerning them, in particular 
an analogue of Kerov positivity conjecture. 
These conjectures hint on existence of some hidden, highly mysterious combinatorial structure behind
Jack symmetric functions.
A natural direction of attack would be to guess and prove an analogue of the Stanley's character formula
\eqref{eq:stanley-general}. 
There are some hints on how this generalization should look like:
for a special value $\gamma=0$ of the deformation parameter
we recover the usual characters of the symmetric groups
and the Stanley character formula \eqref{eq:stanley-general} is available; 
viewed as in Section \ref{sec:stanley-embedding} it involves a summation over \emph{oriented} maps. 
For other special choices $\gamma=\pm \frac{1}{\sqrt{2}}$ an analogue of the Stanley character formula is also known \cite{FeraySniady2011a}, but this time it involves a summation over \emph{all, possibly non-orientable} maps.
Thus it is natural to suspect that the hypothetical Stanley character formula for the general case should
involve some weight which measures ``non-orientability'' of a given map.
Solution of this problem would hopefully shed some light on some old conjectures concerning
the hypothetical combinatorial interpretation of some other quantities 
related to Jack symmetric functions, including
the \emph{matching Jack conjecture} of Goulden and Jackson
\cite{GouldenJackson1996,LaCroix2009}.
Unfortunately, making the right guess for the Stanley formula is not an easy task \cite{DolegaFeraySniady2013}.
In order to tease the Reader, we provide a toy example below in the case of two rectangles.
\begin{dmath*}
\Ch_3^{(\gamma)}(\pp\times \qq)=
p_{1}^{3} q_{1} - 3 p_{1}^{2} q_{1}^{2} + p_{1} q_{1}^{3} + 3 p_{1}^{2}
p_{2} q_{2} + 3 p_{1} p_{2}^{2} q_{2} + p_{2}^{3} q_{2} - 3 p_{1} p_{2}
q_{1} q_{2} - 3 p_{1} p_{2} q_{2}^{2} - 3 p_{2}^{2} q_{2}^{2} + p_{2}
q_{2}^{3} - 3 p_{1}^{2} q_{1} \gamma + 3 p_{1} q_{1}^{2} \gamma - 6
p_{1} p_{2} q_{2} \gamma - 3 p_{2}^{2} q_{2} \gamma + 3 p_{2} q_{2}^{2}
\gamma + 2 p_{1} q_{1} \gamma^{2} + 2 p_{2} q_{2} \gamma^{2} + p_{1}
q_{1} + p_{2} q_{2}.
\end{dmath*}
Maybe Richard Stanley could again help and guess the formula standing behind such computer-generated data?

Finally, to readers who would like to get a different perspective on Kerov polynomials and the Stanley character formula
we recommend an overview article \cite{Sniady2012}.

\bibliographystyle{alpha}

\bibliography{stanley}

\end{document}